\documentclass[12pt]{amsart}
\usepackage{amssymb, amsmath}
\title{On the Equational Artinian Algebras}
\author{P. Modabberi \and M. Shahryari}
\address{ P. Modabberi: Department of Pure Mathematics,  Faculty of Mathematical
Sciences, University of Tabriz, Tabriz, Iran}
\email{p\_modabberi@tabrizu.ac.ir}
\address{M. Shahryari: Department of Pure Mathematics,  Faculty of Mathematical
Sciences, University of Tabriz, Tabriz, Iran}
\email{mshahryari@tabrizu.ac.ir}

\newtheorem {theorem}{Theorem}

\newtheorem{definition}{Definition}
\begin{document}

\maketitle

\begin{abstract}
Equational Artinian algebras were introduced in our previous work: {\em Equational conditions in universal algebraic geometry, to appear in Algebra and Logic, 2015}. In this note, we define the notion of {\em radical topology with respect to an algebra $A$} and using the well-known K\"{o}nig lemma in graph theory, we show that the algebra $A$ is equational Artinian iff this topology is noetherian. This completes the analogy between equational noetherian and equational Artinian algebras. As an application, we prove that every  ultra-power of an equational Artinian algebra is equational Artinian. This provides man examples of equational Artinian groups.
\end{abstract}


\section{Introduction}
Universal algebraic geometry is a new area of modern algebra, whose
subject is basically the study of equations over an arbitrary
algebraic structure $A$. In the classical algebraic geometry $A$ is
a field.  Many articles already published about algebraic geometry
over groups, see \cite{BMR1}, \cite{BMR2}, \cite{BMRom},
\cite{KM}, and \cite{MR}. In an outstanding series of papers, O. Kharlampovich and A. Miyasnikov
developed algebraic geometry over free groups to give
affirmative answer for an old problem of Alfred Tarski concerning
elementary theory of free groups (see \cite{KMTarski} and also \cite{SEL} for the independent solution of Z. Sela).
Also in \cite{KMTarski2}, a problem of Tarski about decidablity of the elementary theory of free groups is solved. Algebraic
geometry  is also developed for algebras
other than groups, so in a series of papers, the systematic study of universal algebraic geometry is
done  by V. Remeslennikov, A. Myasnikov and
E. Daniyarova in \cite{DMR1}, \cite{DMR2}, \cite{DMR3}, and
\cite{DMR4}.

Equational noetherian algebras are among important classes of examples extensively studied in universal algebraic  geometry. These are algebras the Zariski topology in which satisfies descending chain condition for closed subsets. In \cite{ModSH} we studied equational Artinian algebras in analogy to equational noetherian ones. We defined an equational Artinian algebra as an algebra the Zariski topology in which satisfies ascending chain condition for {\em algebraic sets}. It seems that there should be a full symmetry between two similar notions, however, in \cite{ModSH} we couldn't complete description of this symmetry. In this note, we will define {\em radical topology with respect to an algebra $A$} over the set of atomic formulas of the underlying language $\mathcal{L}$ which is a counterpart for the Zariski topology. Then we will show that $A$ is equational Artinian iff this topology satisfies descending chain condition for {\em closed} subsets. As a results, we see that although the property of being equational noetherian translates in terms of Zariski topology, the property of being equational Artinian translates in terms of the corresponding radical-topology. The main tool for our argument is the well-known lemma of K\"{o}nig from graph theory which says that every tree with vertices of finite degrees is finite if the length of any chain is finite. As an application, we prove that every  ultra-power of an equational Artinian algebra is equational Artinian. This provides many examples of equational Artinian groups.

\section{Basic notions}
This section is devoted to a fast review of the basic concepts of the universal algebraic geometry. We suggest \cite{BS}
for reader who is not familiar to the universal algebra. The reader also would use \cite{DMR1}, \cite{DMR2}, \cite{DMR3}, and
\cite{DMR4}, for extended exposition of the universal algebraic geometry.  Our notations here are almost the same as in the above mentioned papers.
For the sake of  simplicity, we define our notions in the coefficient free frame and then one can extend all the definitions and results to the case of non-coefficient free algebraic geometry.

Fix an algebraic language $\mathcal{L}$ and a set of variables $X=\{ x_1, \ldots, x_n\}$. An equation is a pair $(p, q)$ of the elements of the term algebra $T_{\mathcal{L}}(X)$. In many cases, we assume that such an equation is the same as the  atomic formula $p(x_1, \ldots, x_n)\approx q(x_1, \ldots, x_n)$ or $p\approx q$ in short. Hence, in this article the set $At_{\mathcal{L}}(X)$ of atomic formulae in the language $\mathcal{L}$ and the product algebra $T_{\mathcal{L}}(X)\times T_{\mathcal{L}}(X)$ are assumed to be equal.

Any subset $S\subseteq At_{\mathcal{L}}(X)$ is called a {\em system of equations} in the language $\mathcal{L}$. A system $S$ is called {\em consistent} over an algebra $A$, if there is an element $(a_1, \ldots,
a_n)\in A^n$ such that for all equations $(p\approx q)\in S$, the
equality
$$
p^A(a_1, \ldots, a_n)=q^A(a_1, \ldots, a_n)
$$
holds. Otherwise, we say that $S$ is {\em in-consistent} over $A$. Note that, $p^A$ and $q^A$ are the corresponding term functions on $A^n$. A system of equations $S$ is called an ideal,
if it corresponds to a  congruence on $T_{\mathcal{L}}(X)$. For an arbitrary system of
equations $S$, the ideal generated by $S$, is the smallest congruence containing $S$ and it is denoted by $[S]$.

For an algebra $A$ of type $\mathcal{L}$, an element $(a_1, \ldots, a_n)\in A^n$  will be denoted by $\overline{a}$, sometimes. Let  $S$ be  a system of equations. Then the set
$$
V_A(S)=\{ \overline{a}\in A^n: \forall (p\approx q)\in S,\ p^A(\overline{a})=q^A(\overline{a})\}
$$
is called an {\em algebraic set}. It is clear that for any non-empty family $\{ S_i\}_{i\in I}$, we have
$$
V_A(\bigcup_{i\in I}S_i)=\bigcap_{i\in I}V_A(S_i).
$$
So, we define a closed set in $A^n$ to be an arbitrary intersections of finite unions of algebraic sets. Therefore, we obtain a topology on $A^n$,
which is called {\em Zariski topology}.

For any set $Y\subseteq A^n$, we define
$$
\mathrm{Rad}(Y)=\{ (p, q): \forall\ \overline{a}\in Y,\ p^A(\overline{a})=q^A(\overline{a})\}.
$$
It is easy to see that $\mathrm{Rad}(Y)$ is an ideal in the term algebra. Any ideal of this type is called an {\em $A$-radical ideal} or a {\em radical ideal} for short.
Note that any ideal in the term algebra is in fact a radical ideal with respect to some suitable algebra. To see the reason, just note that for any ideal $R$ in the term algebra
$T_{\mathcal{L}}(X)$, if we consider the algebra $B(R)=T_{\mathcal{L}}(X)/R$, then $\mathrm{Rad}_{B(R)}(R)=R$.

It is easy to see that a set $Y$ is algebraic if and only if $V_A(\mathrm{Rad}(Y))=Y$. In the general case, we have $V_A(\mathrm{Rad}(Y))=Y^{ac}$,
see \cite{DMR2}. The {\em coordinate algebra} of a set $Y$ is  the quotient algebra
$$
\Gamma(Y)=\frac{T_{\mathcal{L}}(X)}{\mathrm{Rad}(Y)}.
$$
An arbitrary element of $\Gamma(Y)$ is denoted by $[p]_Y$. We define a function $p^Y:Y\to A$ by the rule
$$
p^Y(\overline{a})=p^A(a_1, \ldots,a_n),
$$
which is a  {\em term function} on $Y$. The  set of all such functions will be denoted by $T(Y)$ and it is naturally an algebra of type $\mathcal{L}$.
It is easy to see that the map $[p]_Y\mapsto p^Y$ is a well-defined isomorphism. So, we have $\Gamma(Y)\cong T(Y)$.

For a system of equation, we can also define the radical $\mathrm{Rad}_A(S)$ to be $\mathrm{Rad}(V_A(S))$. Two systems $S$ and $S^{\prime}$
are called equivalent over $A$, if they have the same set of solutions in $A$, i.e. $V_A(S)=V_A(S^{\prime})$. So, clearly $\mathrm{Rad}_A(S)$ is
the largest system which is equivalent to $S$. Note that $[S]\subseteq \mathrm{Rad}_A(S)$.

\begin{definition}
An algebra $A$ is called equational noetherian, if for any system of equations $S$, there exists a finite subsystem $S_0\subseteq S$, which is
equivalent to $S$ over $A$, i.e. $V_A(S)=V_A(S_0)$.
\end{definition}

Many examples of equational
noetherian algebras are introduced in \cite{DMR2}. Among them are noetherian rings and linear groups over noetherian rings as well as  free
groups. In \cite{DMR2}, it is proved that the next four assertions are equivalent:\\

{\em
i- An algebra $A$ is equational noetherian.\\

ii- For any system $S$, there exists a finite $S_0\subseteq [S]$, such that $V_A(S)=V_A(S_0)$.\\

iii- For any $n$, the Zariski topology on $A^n$ is noetherian, i.e. any descending chain of closed subsets terminates.\\

iv- Any chain of coordinate algebras and epimorphisims
$$
\Gamma(Y_1)\to \Gamma(Y_2)\to \Gamma(Y_3)\to \cdots
$$
terminates}.\\

So, in the case of equational noetherian algebras, any closed set in $A^n$ is equal to a minimal finite union of {\em irreducible} algebraic sets which is unique up to a permutation. Note that a set is called irreducible, if it has no proper finite covering consisting of closed sets.
The following theorem is proved in \cite{DMR2}.

\begin{theorem}
Let $A$ be an equational noetherian algebra. Then the following algebras are also equational noetherian:\\

i- any subalgebra and filter-power of $A$.\\

ii- any coordinate algebra over $A$.\\

iii- any fully residually  $A$-algebra.\\

iv- any algebra belonging to the quasi-variety generated by $A$.\\

v- any algebra universally equivalent to $A$.\\

vi- any limit algebra over $A$.\\

vii- any finitely generated algebra defined by a complete atomic type in the universal theory of $A$ or in the set of quasi-identities of $A$.
\end{theorem}

\section{Equational Artinian algebras}

We say that an algebra $A$ is {\em equational Artinian} if every
ascending chain of algebraic sets over $A$ terminates. We first, review some basic properties of equational Artinian algebras (see \cite{ModSH} for the proofs). One can ask about the existence of  an equational condition, equivalent to being equational Artinian. We proved that the correct condition is not in terms of equations, but rather it can be formulated in terms of radical ideals. We showed that  $A$ is equational Artinian, iff for any $n$ and $E\subseteq A^n$, there exists a finite subset $E_0\subseteq E$ such that
$$
\mathrm{Rad}(E)=\mathrm{Rad}(E_0).
$$
Note that this condition is in some sense the dual condition of being equational noetherian. Recall the following definition of the radical ideals.

\begin{definition}
Let $A$ be an algebra and $E\subseteq A^n$, for some $n$. Then $\mathrm{Rad}(E)$ is  called an $A$-radical ideal of the term algebra
$T_{\mathcal{L}}(x_1, \ldots, x_n)$.
\end{definition}

The collection of all $A$-radical ideals of the term algebra, is a subbasis of closed sets for a topology on the set of atomic formulas $At_{\mathcal{L}}(x_1, \ldots, x_n)$ which we call it {\em the radical topology with respect to $A$}. As in the case of the Zariski topology, closed sets of this new space are arbitrary intersections of finite unions of $A$-radical ideals.
The next theorem provides some equivalent conditions for the property of being equational Artinian. The conditions i, ii and vi are not new (they are obtained in \cite{ModSH}), so we will concentrate on iv and v. Recall that we say that topological space is {\em contra-compact} if every covering of it by closed sets has a finite subcover.

\begin{theorem}
For an algebra $A$, the following conditions are equivalent;\\

i-  For any $n$ and $E\subseteq A^n$, there exists a finite subset $E_0\subseteq E$ such that
$$
\mathrm{Rad}(E)=\mathrm{Rad}(E_0).
$$

ii- Every descending chain of $A$-radical ideals terminates.\\

iii- $A$ is equational Artinian.\\

iv- For any $n$, the radical topology with respect to $A$ on $At_{\mathcal{L}}(x_1, \ldots, x_n)$ is noetherian.\\

v- For any $n$, every subset of $At_{\mathcal{L}}(x_1, \ldots, x_n)$ is compact.\\

vi- For any $n$, every subset of $A^n$ is contra-compact.
\end{theorem}

\begin{proof}
We only show that iv and v are equivalent to the property  of being equational Artinian. To prove the equivalence of iii and iv,
suppose that the algebra $A$ is equationally Artinian. We assume that $B$ is the set of all $A$-radicals corresponding to the  subsets of $A^n$. So $B$ satisfies the descending chain condition. Let $B_1$ be the set of all finite unions of the radicals in $B$ and let $B_2$    be the set of all intersections of sets in $B_1$. By the definition of radical-topology,  $B_2$ is just the set of  the closed subsets of this topology. First we prove that $B_1$ satisfies the descending chain condition. Suppose that
$$
M_1=\mathrm{Rad}(E_1)\cup \ldots\cup \mathrm{Rad}(E_{m}), \;\;\;\;M_2=\mathrm{Rad}(E_1^{\prime})\cup \ldots\cup \mathrm{Rad}(E_k^{\prime})
$$
are sets in $B_1$ and $M_2\subset M_1$. For every $i\leq m$ and $j\leq k$, we have
$\mathrm{Rad}(E_i)\cap \mathrm{Rad}(E_j^{\prime})\subset \mathrm{Rad}(E_i)$. Hence we can gain a tree with  root vertex $\mathrm{Rad}(E_i)$ and with a unique edge from the root to every subset $\mathrm{Rad}(E_i)\cap \mathrm{Rad}(E_j^{\prime})$.
We consider the following strictly descending chain of subsets of $B_1$:
$$
M_1\supset M_2 \supset M_3 \supset \ldots .
$$
As we mentioned, we obtain  a tree such that each vertex  is a finite intersection of sets in $B$, hence each vertex is in $B$ itself.  Since each vertex is connected to only finite other vertexes, so each vertex has a finite degree. So,  every path corresponds to a strictly descending chain of radicals and since $A$ is equationally Artinian, so the path is finite. By K\"{o}nig's lemma this implies that the graph is finite. Therefore the above chain is also finite. So $B_1$ satisfies the descending chain condition and is closed under finite intersection.

Now  we prove that $B_2$ satisfies the descending chain condition too. Suppose $\bigcap_{i=1}^{\infty}R_i$ is an infinite intersection of subsets of $B_1$. Then we have the following chain:
$$
R_1\supseteq R_1\cap R_2 \supseteq  R_1\cap R_2\cap R_3\supseteq \ldots.
$$
Since $B_1$ satisfies descending chain condition and is closed under finite intersection, so the chain terminates. Therefore
$$
\exists m  \;\; R_1\cap R_2\cap \ldots \cap R_m=\bigcap_{i=1}^{\infty}R_i.
$$
Therefore,  every infinite intersection of subsets of $B_1$ is in fact a finite intersection in $B_1$ and so it belongs to $B_1$. Consequently we have $B_2=B_1$ and hence  it satisfies the descending chain condition. This shows that the radical topology on  $At_{\mathcal{L}}(x_1, \ldots, x_n)$ is noetherian. Clearly, if we assume that the radical topology is noetherian, then every descending chain of $A$-radical ideals terminates and so we obtain iii.

Now we prove that v is equivalent to iii. Suppose any  subset of $At_{\mathcal{L}}(x_1, \ldots, x_n)$ is compact. For an arbitrary $E\subseteq A^n$ we have
$$
\mathrm{Rad}(E)=\bigcap_{\overline{a}\in E} \mathrm{Rad}(\overline{a}).
$$
Therefore
\begin{eqnarray*}
At_{\mathcal{L}}(x_1, \ldots, x_n)\setminus \mathrm{Rad}(E)&=&At_{\mathcal{L}}(x_1, \ldots, x_n)\setminus \bigcap_{\overline{a}\in E} \mathrm{Rad}(\overline{a})\\
 &=&\bigcup_{\overline{a}\in E}(At_{\mathcal{L}}(x_1, \ldots, x_n)\setminus   \mathrm{Rad}(\overline{a})).
\end{eqnarray*}

Since any  subset of $At_{\mathcal{L}}(x_1, \ldots, x_n)$  is compact, there is a finite number of points $\overline{a}_1, \ldots, \overline{a}_m$, such that
$$
At_{\mathcal{L}}(x_1, \ldots, x_n)\setminus \mathrm{Rad}(E)=\bigcup_{i=1}^{m}(At_{\mathcal{L}}(x_1, \ldots, x_n)\setminus \mathrm{Rad}(\overline{a}_i)).
$$
Hence
$$
\mathrm{Rad}(E)=\bigcap_{i=1}^{m}\mathrm{Rad}(\overline{a}_i)=\mathrm{Rad}(E_{0}),
$$
where $E_0=\{  \overline{a}_1, \ldots, \overline{a}_m\}$. So $A$ is equational Artinian.
Conversely, suppose that $A$ is  equational Artinian. Note that by this assumption, every closed set in radical topology is a finite union of $A$-radical ideals. This is because by iv, every descending chain of closed sets in $At_{\mathcal{L}}(x_1, \ldots, x_n)$ terminates. Let $S\subseteq At_{\mathcal{L}}(x_1, \ldots, x_n)$ and
$$
S\subseteq \bigcup_{i}(At_{\mathcal{L}}(x_1, \ldots, x_n)\setminus \mathrm{Rad}(\overline{a}_i))
$$
be an open covering of $S$ (by the point we just mentioned, every open covering of $S$ has this form). Then
\begin{eqnarray*}
S&\subseteq& At_{\mathcal{L}}(x_1, \ldots, x_n)\setminus \bigcap_{i}\mathrm{Rad}(\overline{a}_i)\\
&=&At_{\mathcal{L}}(x_1, \ldots, x_n)\setminus \mathrm{Rad}(E).
\end{eqnarray*}

Now, since $A$ is equational Artinian, there is a finite $E_0\subseteq E$ such that $\mathrm{Rad}(E)=\mathrm{Rad}(E_0)$,  so
\begin{eqnarray*}
S&\subseteq& At_{\mathcal{L}}(x_1, \ldots, x_n)\setminus \mathrm{Rad}(E_0)\\
&=&\bigcup_{i=1}^{m}(At_{\mathcal{L}}(x_1, \ldots, x_n)\setminus \mathrm{Rad}(\overline{a}_i)).
\end{eqnarray*}

\end{proof}

This theorem, completes the analogy between two {\em dual} properties of being equational noetherian and equational Artinian. As a result, every closed subset of $At_{\mathcal{L}}(x_1, \ldots, x_n)$ is a unique union of finitely many irreducible sets. These irreducible sets are necessarily have the form $\mathrm{Rad}(Y)$, for some $Y\subseteq A^n$. Such an algebraic set $Y$ will be called {\em large}. So, if $A$ is equational Artinian, then every algebraic set $Y\subseteq A^n$ can be written uniquely as a finite intersection $Y=\bigcap_{i=1}^mY_i$, such that every $Y_i$ is a large algebraic set.

As another application of the previous theorem, we show that every ultrapower of an equational Artinian algebra is equational Artinian.

\begin{theorem}
Let $A$ be equational Artinian and $I$ be a set. Let $\mathcal{U}$ be an ultrafilter over $I$. Then the corresponding ultrapower $A^I/\mathcal{U}$ is also equational Artinian.
\end{theorem}

\begin{proof}
Consider an arbitrary equation $(p\approx q)\in At_{\mathcal{L}}(x_1, \ldots, x_n)$ and let $B= A^I/\mathcal{U}$. We have
$$
\mathrm{Rad}_B(p\approx q)=\{ (p^{\prime}\approx q^{\prime}):\ B\vDash \forall \overline{x}(p(\overline{x})\approx q(\overline{x})\to p^{\prime}(\overline{x})\approx q^{\prime}(\overline{x}))\}.
$$
Since $A$ and $B$ are elementary equivalent, so we must have $\mathrm{Rad}_B(p\approx q)=\mathrm{Rad}_A(p\approx q)$. Now, for all $S\subseteq At_{\mathcal{L}}(x_1, \ldots, x_n)$, we have
\begin{eqnarray*}
\mathrm{Rad}_B(S)&=& \bigcap_{(p\approx q)\in S}\mathrm{Rad}_B(p\approx q)\\
                 &=& \bigcap_{(p\approx q)\in S}\mathrm{Rad}_A(p\approx q)\\
                 &=&\mathrm{Rad}_A(S).
\end{eqnarray*}
This shows that the radical topologies on  $At_{\mathcal{L}}(x_1, \ldots, x_n)$ with respect to $A$ and $B$ are the same. So by Theorem 2, the algebra $B$ is equational Artinian.
\end{proof}

As an example, let $G$ be  finite group and $A=G\times G\times \cdots$ be a direct product of infinitely many copies of $G$. We know that $A$ is equational Artinian and so for any set $I$ and any ultrafilter $\mathcal{U}$ over $I$ the group $A^I/\mathcal{U}$ is also equational Artinian.

We close this note by a version of Theorem 1 for equational Artinian algebras which can be proved using Theorem 3 and similar arguments of \cite{DMR2}.

\begin{theorem}
Let $A$ be an equational Artinian algebra. Then the following algebras are also equational Artinian:\\

i- any subalgebra and ultrapower of $A$.\\

ii- any coordinate algebra over $A$.\\

iii- any fully residually  $A$-algebra.\\

iv- any algebra universally equivalent to $A$.\\

v- any limit algebra over $A$.\\

vii- any finitely generated algebra defined by a complete atomic type in the universal theory of $A$.
\end{theorem}

\end{document}